\documentclass[12pt,a4paper]{amsart}
  \usepackage{enumerate}
  \usepackage{pb-diagram}
  \usepackage{microtype}
  \usepackage{appendix}
  \usepackage{url}
  \usepackage{amssymb, pb-diagram}
  \usepackage[all]{xy}
  \usepackage{graphicx} 
  \usepackage[dvipsnames]{xcolor}
  \usepackage{chapterbib}
  \usepackage{epsfig}
  \usepackage{amsfonts}
  \usepackage{amssymb}
  \usepackage{amsmath}   
  \usepackage{amsthm}
  \usepackage{color}
  \usepackage{dsfont}
  \usepackage{latexsym}
  \usepackage{hyperref}
  \usepackage{mathabx}
  \usepackage{graphicx}
  \usepackage{tikz-cd}
  \usepackage{mathabx}
  \usepackage{enumitem}
  \usepackage{lipsum}
  \usepackage{adjustbox}

  \makeatletter
  \let\@wraptoccontribs\wraptoccontribs
  \makeatother

  \input xy
  \xyoption{all}
  \usepackage{pgf, tikz}
  \usetikzlibrary{arrows,shapes}

  \renewcommand{\epsilon}{\varepsilon}
  \renewcommand{\setminus}{\smallsetminus}
  \renewcommand{\emptyset}{\varnothing}

  \numberwithin{equation}{section}
  \newtheorem{theorem}[equation]{Theorem}
  \newtheorem{observation}[equation]{Observation}
  \newtheorem{proposition}[equation]{Proposition}
  \newtheorem{corollary}[equation]{Corollary}
  \newtheorem{lemma}[equation]{Lemma}
  
  \theoremstyle{definition}

  \newtheorem{definition}[equation]{Definition}

  \newtheorem{remark}[equation]{Remark}

  \newcommand{\Link}{\lk}

  \DeclareMathOperator{\join}{\ast}

  \newcommand{\cohom}[3]{H^{{\raise1pt\hbox{$\scriptstyle#1$}}}(#2\>\!,#3)}
  \newcommand{\tatecohom}[3]%
  {\widehat H^{{\raise1pt\hbox{$\scriptstyle#1$}}}(#2\>\!,#3)}

  \newcommand{\Cohom}[3]%
  {H^{{\raise1pt\hbox{$\scriptstyle#1$}}}\big(#2\>\!,#3\big)}
  \newcommand{\Tatecohom}[3]%
  {\widehat H^{{\raise1pt\hbox{$\scriptstyle#1$}}}\big(#2\>\!,#3\big)}

  \newcommand{\homol}[3]{H_{{\lower1pt\hbox{$\scriptstyle#1$}}}(#2\>\!,#3)}
  \newcommand{\homolog}[2]{H_{{\lower1pt\hbox{$\scriptstyle#1$}}}(#2)}

  \newcommand{\da}{\downarrow}

  \newcommand{\Map}{\operatorname{Map}}

  \newcommand{\frakF}{\mathfrak{F}}
  \newcommand{\frakG}{\mathfrak{G}}

  \catcode`\@=11

  \newcommand{\calO}{\mathcal O}
  \newcommand{\B}{\mathcal B}

  \newcommand{\ModOFG}{\mathop{{\operator@font
        Mod\text{-}}\calO_{\frakF}G}}
  \newcommand{\OFGMod}{\mathop{\calO_{\frakF}G\text{-}{\operator@font
        Mod}}}
  \newcommand{\ModOGG}{\mathop{{\operator@font
        Mod\text{-}}\calO_{\frakG}G}}
  \newcommand{\OGGMod}{\mathop{\calO_{\frakG}G\text{-}{\operator@font
        Mod}}}

  \DeclareMathOperator{\PMap}{PMap}
  
  \DeclareMathOperator{\lk}{lk}

  \newcommand{\base}{\rm{base}}

  \newcommand{\Sym}{\operatorname{Sym}}

  \newcommand{\Fall}{\frakF_{\operator@font all}}
  \newcommand{\Ffin}{\frakF_{\operator@font fin}}
  \newcommand{\Fvc}{\frakF_{\operator@font vc}}
  \newcommand{\Fic}{\frakF_{\operator@font ic}}
  \newcommand{\Ffg}{\frakF_{\operator@font fg}}
  \newcommand{\Fpc}{\frakF_{\operator@font pc}}
  \newcommand{\Fab}{\frakF_{\operator@font ab}}
  \newcommand{\Fvpc}{\frakF_{\operator@font vpc}}
  \newcommand{\Fvab}{\frakF_{\operator@font vab}}

  \catcode`\@=12

  \DeclareMathOperator{\id}{id}

  \renewcommand{\coprod}%
  {\mathop{\rotatebox[origin=c]{180}{$\displaystyle\prod$}}\limits}

  \numberwithin{equation}{section}

  \newcommand{\notion}[1]{\emph{#1}}

  \makeatletter
  
  \providecommand{\IfEmptyThenElse}[1]{%
    \begingroup
    \def\tempa{#1}%
    \def\tempb{}%
    \ifx\tempa\tempb
      \def\next{\@firstoftwo}%
    \else
      \def\next{\@secondoftwo}%
    \fi
    \expandafter\endgroup\next
  }
  \newcommand{\IfNonEmpty}[2]{\IfEmptyThenElse{#1}{}{#2}}

  \newlength{\vshrinkamount}
  \newcommand{\vshrink}[2][0pt]{%
    \begingroup
    \setlength\fboxrule{0pt}%
    \setlength\fboxsep{-\vshrinkamount}%
    \framebox[\width]{$#2$}%
    \endgroup
  }
  \newcommand{\OldParentheses}[1]{\left(\vshrink{#1}\right)}
  \newcommand{\ParenthesesDouble}[1]{\left(\left(\vshrink{#1}\right)\right)}
  \newcommand{\Brakets}[1]{\left[\vshrink{#1}\right]}
  \newcommand{\BraketsDouble}[1]{\left[\left[\vshrink{#1}\right]\right]}

  \newcommand{\EvalKerning}{\mathchoice{%
      \negthinspace
    }{%
      \negthinspace
    }{%
    }{%
    }}
  \newcommand{\EvalAt}{\EvalKerning\OldParentheses}
  \newcommand{\EvalAtAt}{\EvalKerning\ParenthesesDouble}
  \newcommand{\EvalAd}{\EvalKerning\Brakets}
  \newcommand{\EvalAdAd}{\EvalKerning\BraketsDouble}

  \newcommand{\OptExponent}[1][]{\IfNonEmpty{#1}{^{#1}}}
  \newcommand{\OptParExponent}[1][]{\IfNonEmpty{#1}{^{\left(#1\right)}}}
  \newcommand{\OptIndex}[1][]{\IfNonEmpty{#1}{_{#1}}}
  \newcommand{\OptArg}[2]{#1\IfNonEmpty{#2}{\EvalAt{#2}}}
  \newcommand{\OptIndexOptPar}[1][]{\IfNonEmpty{#1}{_{#1}}\OptParentheses}
  \newcommand{\OptParentheses}[1][]{\IfNonEmpty{#1}{\EvalKerning\left(#1\right)}}
  \newcommand{\OptExponentParentheses}[1][]{\IfNonEmpty{#1}{^{#1}}\EvalAt}
  \newcommand{\OptExponentParenthesesDouble}[1][]{\IfNonEmpty{#1}{^{#1}}\EvalAtAt}
  \newcommand{\OptExponentBrakets}[1][]{\IfNonEmpty{#1}{^{#1}}\EvalAd}
  \newcommand{\OptExponentBraketsDouble}[1][]{\IfNonEmpty{#1}{^{#1}}\EvalAdAd}
  \newcommand{\OptIndexParentheses}[1][]{\IfNonEmpty{#1}{_{#1}}\OptExponentParentheses}
  \newcommand{\OptIndexParenthesesDouble}[1][]{\IfNonEmpty{#1}{_{#1}}\OptExponentParenthesesDouble}
  \newcommand{\OptIndexBrakets}[1][]{\IfNonEmpty{#1}{_{#1}}\OptExponentBrakets}
  \newcommand{\OptIndexBraketsDouble}[1][]{\IfNonEmpty{#1}{_{#1}}\OptExponentBraketsDouble}
  \newcommand{\OptIndexParExponent}[1][]{\IfNonEmpty{#1}{_{#1}}\OptParExponent}
  \newcommand{\OptIndexExponent}[1][]{\IfNonEmpty{#1}{_{#1}}\OptExponent}
  \newcommand{\DoubleArg}[2]{\EvalAt{#1,#2}}
  \newcommand{\OptExponentDoubleArg}[1][]{\IfNonEmpty{#1}{_{#1}}\DoubleArg}
  \newcommand{\OptIndexDoubleArg}[1][]{\IfNonEmpty{#1}{_{#1}}\OptExponentDoubleArg}
  
  \newcommand{\UnVar}[1]{%
    \let#1\undefined
    \expandafter\let\csname\expandafter\@gobble\string#1Of\endcsname\undefined
    \expandafter\let\csname\expandafter\@gobble\string#1OfOf\endcsname\undefined
    \expandafter\let\csname\expandafter\@gobble\string#1Ad\endcsname\undefined
    \expandafter\let\csname\expandafter\@gobble\string#1AdAd\endcsname\undefined
  }
  \newcommand{\newvariable}[2]{%
    \newcommand{#1}[1][]{#2\IfNonEmpty{##1}{_{##1}}\OptExponent}
    \expandafter\newcommand\csname\expandafter\@gobble\string#1Of\endcsname{#1\OptIndexParentheses}
    \expandafter\newcommand\csname\expandafter\@gobble\string#1OfOf\endcsname{#1\OptIndexParenthesesDouble}
    \expandafter\newcommand\csname\expandafter\@gobble\string#1Ad\endcsname{#1\OptIndexBrakets}
    \expandafter\newcommand\csname\expandafter\@gobble\string#1AdAd\endcsname{#1\OptIndexBraketsDouble}
  }
  \newcommand{\renewvariable}[2]{%
    \renewcommand{#1}[1][]{#2\IfNonEmpty{##1}{_{##1}}\OptExponent}
    \expandafter\renewcommand\csname\expandafter\@gobble\string#1Of\endcsname{#1\OptIndexParentheses}
    \expandafter\renewcommand\csname\expandafter\@gobble\string#1Ad\endcsname{#1\OptIndexBrakets}
  }

  \newcommand{\getlastargs}[2][]{%
    \def\tempb{#1}%
    \def\empty{}%
    \ifx\EndSymbol\empty
      \ifx\tempb\empty
        \ifx\tempa\empty
          \TheSymbol #2%
        \else
          \TheSymbol_{\tempa}#2%
        \fi
      \else
        \TheSymbol_{\tempa}^{#1}#2%
      \fi
    \else
      \ifx\tempb\empty
        \ifx\tempa\empty
          \TheSymbol #2\EndSymbol
        \else
          \TheSymbol #2\EndSymbol_{\tempa}%
        \fi
      \else
        \TheSymbol #2\EndSymbol_{\tempa}^{#1}%
      \fi
    \fi
    \endgroup
  }

  \makeatother

  \newcommand{\mapcolon}{\colon}

  \newcommand{\MaxOf}[2][]{%
    \IfEmptyThenElse{#1}{
      \max\left(\vshrink{#2}\right)
    }{
      \max_{#1}\left(\vshrink{#2}\right)
    }
  }

  \newcommand{\MinOf}[2][]{%
    \IfEmptyThenElse{#1}{
      \min\left(\vshrink{#2}\right)
    }{
      \min_{#1}\left(\vshrink{#2}\right)
    }
  }

  \newvariable{\TheGraph}{\Gamma}
  \newvariable{\TheSurface}{Z}
  \newvariable{\NumBlue}{m}
  \newvariable{\NumRed}{n}
  \newvariable{\NumBlack}{r}
  \newvariable{\ArcCx}{\mathcal{C}}
  \newvariable{\Connectivity}{\nu}
  \newvariable{\TheArc}{\alpha}
  \newvariable{\AltArc}{\beta}
  \newvariable{\TheDim}{d}
  \newvariable{\CutSurface}{{\TheSurface}'}
  \newvariable{\TheConnectivity}{l}
  \newvariable{\AltSurface}{\TheSurface^*}
  \newvariable{\AltBlue}{\NumBlue^*}
  \newvariable{\AltRed}{\NumRed^*}
  \newvariable{\AltBlack}{\NumBlack^*}
  \newvariable{\TheSpace}{K}
  \newvariable{\AltSpace}{L}
  \newvariable{\IntSpace}{M}
  \newvariable{\ThePoint}{P}
  \newvariable{\BluePoint}{P}
  \newvariable{\RedPoint}{Q}
  \newvariable{\Homotopy}{\pi}
  \newvariable{\TheMap}{\varphi}
  \newvariable{\AltMap}{\psi}
  \newvariable{\TheSphere}{E}
  \newvariable{\SimplexDim}{k}
  \newvariable{\TheVertex}{v}
  \newvariable{\AltVertex}{w}
  \newvariable{\IntersectionSet}{I}
  \newvariable{\NumCrossings}{\#}
  \newvariable{\ClosestPoint}{N}
  \newvariable{\TheLinkSphere}{L}
  \newvariable{\CommonLink}{K}
  \let\Link\undefined
  \newvariable{\Link}{\operatorname{lk}}
  \newcommand{\CardOf}[1]{|#1|}

  \newvariable{\TheJoinCx}{Y}
  \newvariable{\TheBaseCx}{X}
  \newvariable{\TheJoinPr}{\pi}
  
  \newvariable{\HomotopyGroup}{\pi}
  
  \newcommand{\DimOf}[1]{\operatorname{dim}(#1)}
  \newvariable{\TheSimplicialMap}{f}
  \newvariable{\TheSimplicialCx}{X}
  \newvariable{\AltSimplicialCx}{Y}
  \newvariable{\AltDim}{n}
  \newvariable{\TheSimplex}{\sigma}
  \newvariable{\AltSimplex}{\tau}
  
  \newvariable{\TheGroup}{G}
  \newvariable{\TheGcx}{X}
  \newvariable{\TheCrVal}{s}
  \newcommand{\DescLinkOf}[1]{\lk^{\da}(#1)}
  
  \newvariable{\PieceCx}{\mathcal{P}}
  \newvariable{\CmpSurface}{Z}
  \newvariable{\StillAdmissible}{Q'}
  \newvariable{\Genus}{g}
  \newvariable{\TheHandle}{T}
  \newvariable{\TheTether}{\alpha}
  \newvariable{\TheBdCmp}{b}
  \newcommand{\BoundaryOf}[1]{\partial #1}

  \newvariable{\ForgetTethers}{\phi}
  \newvariable{\Surface}{S}
  \newvariable{\BdCmps}{Q}
  \newvariable{\THcx}{\mathcal{TH}}
  \newvariable{\THinj}{\mathcal{TH}_1}
  \newvariable{\Hcx}{\mathcal{H}}
  \newvariable{\TargetConn}{k}
  \newvariable{\TheGenus}{g}
  \newvariable{\NumBdComp}{n}
  \newvariable{\OurGroup}{\mathcal{B}}
  \newvariable{\NumStrands}{n}
  \newvariable{\TheType}{d}
  \newcommand{\FType}[1][]{\ensuremath{\text{F}_{#1}}}
  \newcommand{\FPType}[1][]{\ensuremath{\text{FP}_{#1}}}
  
  \newvariable{\TheCubeCx}{\mathfrak{X}}
  \newvariable{\TheMorseFct}{h}
  \newvariable{\NatNumbers}{\mathbb{N}}
  \newvariable{\RealNumbers}{\mathbb{R}}

  \newcommand{\RottenCoreOf}[1]{\bar{#1}}
  \newcommand{\faceof}{\leq}
  \newcommand{\GoodScxOf}[1]{#1^{\mathrm{good}}}
  \newvariable{\BadSimplex}{\sigma}
  \newvariable{\GoodLink}{\operatorname{lk}^{\mathrm{good}}}
  \newcommand{\monorightarrow}{\hookrightarrow}

  \newvariable{\Homology}{H}
  
\begin{document}


  \title[Surface Houghton groups]{Surface Houghton groups}


  \author[Aramayona]{Javier Aramayona}
  \address{Instituto de Ciencias Matem\'aticas (ICMAT). Nicol\'as Cabrera, 13--15. 28049, Madrid, Spain}
  \thanks{J.\,A. was supported by grant PID2021-126254NB-I00 and by the Severo Ochoa award CEX2019-000904-S, funded byr MCIN/AEI/10.13039/501100011033.}

  \author[Bux]{Kai-Uwe Bux}
  \address{Fakultät für Mathematik.
    Universität Bielefeld.
    Postfach 100131.
    Universitätsstraße 25.
    D-33501 Bielefeld, Germany}
  \thanks{K.-U.\,B. is supported by the Deutsche Forschungsgemeinschaft (DFG, German Research Foundation) via the grant SFB-TRR 358/1 2023 — 491392403}

  \author[Kim]{Heejoung Kim}
  \address{Department of Mathematics. Ohio State University. 231 West 18th Avenue, Columbus, OH, USA}
  \thanks{}

  \author[Leininger]{Christopher J. Leininger}
  \address{Rice University
    Math Department -- MS 136.
    P.O. Box 1892.
    Houston, TX 77005-1892, USA.}
  \thanks{C.L. was partially supported by NSF grant DMS-2106419.}

  \date{\today}


  \begin{abstract} 
For every $n\ge 2$, the {\em surface Houghton group} $\mathcal B_n$ is defined as the asymptotically rigid mapping class group  of a surface with exactly $n$ ends, all of
them non-planar. The groups $\mathcal B_n$ are analogous to, and in fact contain, the braided Houghton groups. These groups also arise naturally in topology: every monodromy homeomorphisms of a fibered component of a depth-1 foliation of closed 3-manifold is conjugate into some $\mathcal B_n$.
As countable mapping class groups of infinite type surfaces, the groups $\B_n$ lie somewhere between classical mapping class groups and big mapping class groups.  
We initiate the study of surface Houghton groups proving, among other things, that $\B_n$ is of type $\FType[n-1]$, but not of type $\FPType[n]$, analogous to the braided Houghton groups.  

  \end{abstract}

  \maketitle

  \section{Introduction and results} 
  For $n\ge 2$, denote by $\Sigma_n$ the connected orientable surface
  with exactly $n$ ends, all of them non-planar. We view $\Sigma_n$ as
  constructed by first gluing a torus with two boundary components
  (called a {\em piece}) to every connected component of a sphere with
  $n$ boundary components, and then inductively gluing a piece to every boundary component of the
  surface from the previous step.
 The {\em surface Houghton group}, $\mathcal B_n$, is the subgroup of the mapping class group $\Map(\Sigma_n)$
  whose elements {\em eventually} send pieces to 
  pieces, in a {\em trivial} manner; see Section \ref{sec:prelim} for a
  precise definition.

Viewed in this light, the groups $\mathcal B_n$ are natural analogs of the asymptotic mapping class groups of Cantor surfaces considered in \cite{FK04,FK09,AF17,GLU20,ABF+21}. In addition, these groups are closely related to Houghton groups and their {\em braided} relatives \cite{Deg00}. In fact, using Funar's description  of the braided Houghton group ${\rm br}H_n$ as an asymptotically rigid mapping class group \cite{Funar07}, in Remark \ref{rmk:subgroup} below we briefly explain how to see that  ${\rm br}H_n$ may be realized as a subgroup of $\mathcal B_n$.

Beyond this analogy, the groups $\B_n$ arise naturally in the study of depth-1 foliations of closed $3$--manifolds. 
More precisely, the non-product components are mapping tori of {\em end-periodic homeomorphisms} (see \cite{Fen97}), and every such homeomorphism is conjugate into some $\B_n$ (see \cite[Corollary~2.9]{FKLL21}).

  \subsection{Finiteness properties}
  Recall that a group is of type $\FType[\TheType]$ if it has a classifying
  space with finite $\TheType$-skeleton, and that it is of type $\FPType[\TheType]$ 
  if the integers, regarded as a trivial module over the group, have a projective resolution that is of finite type in dimensions up to $\TheType$. Such a resolution
  can be obtained for a group of type $\FType[\TheType]$ by using the cellular chain complex of the universal cover of a classifying space. Thus, \FType[\TheType]
  implies \FPType[\TheType].

In \cite{GLU20}, Genevois--Lonjou--Urech proved that the braided Houghton group ${\rm br} H_n$ is of type \FType[n-1] but not of type \FPType[n].  
  Our first result is the analog of this result in our setting.
  
  \begin{theorem}\label{thm:finprop}
    $\OurGroup[\NumStrands]$ is of type $\FType[\NumStrands-1]$ but not of type $\FPType[\NumStrands]$.
  \end{theorem}
  In order to prove Theorem
  \ref{thm:finprop}, and as is often the case with this type of result, we
  will make use of a classical criterion of Brown \cite{Br87},
  expressed through the language of discrete Morse theory; see Section
  \ref{sec:brown} for details. More precisely, for each $n\ge 2$ we will
  construct a finite-dimensional contractible cube
  complex on which $\mathcal B_n$ acts, and an invariant discrete Morse
  function on the complex, such that the descending links are highly
  connected. In the construction of the complex as well as in the analysis
  of descending links, we make heavy use of methods developed in
  \cite{ABF+21}.

  \subsection{Abelianization}
  A well-known theorem of Powell \cite{Po} asserts that
  the {\em pure} mapping class group $\PMap(S)$ of a finite-type surface
  $S$ of genus at least three has trivial abelianization. For the surfaces $\Sigma_n$,
  a result of Patel, Vlamis, and the first author \cite[Corollary
  6]{APV20} implies that $H^1(\PMap(\Sigma_n), \mathbb Z) \cong
  \mathbb Z^{n-1}$.  More generally, the abelianization of (pure) mapping class groups of infinite type surfaces is more complicated; see \cite{DD,DP,MT,PW,V}.
  Using \cite{APV20}, we compute the abelianization $\B_n^{ab}$ of $\B_n$, as well as
  that of their pure counterparts $P\B_n$.

  \begin{theorem}\label{abelianizations}
    For all $n \geq 2$, $\B_n^{ab} = \{0\}$ and $P\B_n^{ab} \cong
    \mathbb Z^{n-1}$.
  \end{theorem}

From the description of the abelianizations, we also describe all finite quotients of $\B_n$, proving that they are highly constrained; see Proposition \ref{prop:quotients}. This has the following consequence.

\begin{corollary} \label{cor:braidedsurface}
       For all $m, n\in \mathbb N$, the groups $\B_n$ and ${\rm br}H_m$ are not commensurable. 
\end{corollary}

\subsection{Marking graphs}

By Theorem~\ref{thm:finprop}, $\mathcal B_n$ is finitely generated for all $n \geq 2$, and so it has a well-defined coarse geometry. In the finite type setting, a quasi-isometric model for the mapping class group which has proven quite useful is Masur and Minsky's marking graph \cite{MM00}. For the surfaces $\Sigma_n$, the marking graph is no longer a good model for several reasons. A trivial issue is that it is disconnected, and the orbit of a marking may lie in different components. There are $\mathcal B_n$--invariant components, but even these are problematic since they are no longer locally finite. Worse, the orbit map to any such component fails to be a quasi-isometric embedding (see Section~\ref{S:markings}). On the other hand, there are (many) locally finite subgraphs which do serve as quasi-isometric models, by the following and the Milnor-\v{S}varc Lemma (see e.g.~\cite{BHNPC}).
\begin{theorem} \label{T:sub marking graph}
For all $n \geq 2$, there exists locally finite subgraphs of the marking graph on which $\mathcal B_n$ acts cocompactly. For $n \geq 3$, any marking $\mu$ is contained in such a subgraph.
\end{theorem}
The final statement is not quite true for $n = 2$ since there are markings on $\Sigma_2$ with infinite stabilizer in $\mathcal B_n$.

  \medskip

  \noindent{\bf Plan of the paper.}
  Section~\ref{sec:prelim} contains the relevant background on surfaces, mapping class groups, and the definition of $\mathcal B_n$, and in Section~\ref{sec:brown} we recall a classical result of
  Brown \cite{Br87} about finiteness properties of groups, and describe the relevant tools in our setting. Section
  \ref{sec:cubecx} constructs a contractible cube complex on which $\B_n$ acts nicely, and which is similar to that of
  \cite{GLU20,ABF+21}. In Section \ref{sec:firstpart} we will prove Theorem \ref{thm:finprop}.  We then prove Theorem~\ref{abelianizations} and Corollary~\ref{cor:braidedsurface} in Section
  \ref{sec:abelianization} and  Theorem~\ref{T:sub marking graph} in Section~\ref{S:markings}.

  \medskip

  \noindent{\bf Acknowledgements.}
  The authors would like to thank Anne Lonjou for helpful conversations. 

  \section{Surfaces}\label{sec:prelim}
  Throughout this paper, all surfaces will be assumed to be connected, orientable and second-countable. A surface is said to have
  {\em finite type} if its fundamental group is finitely generated;
  otherwise, it has {\em infinite type}.
  In this paper, the primary surfaces of infinite type we will consider are those that have a finite number of ends, all non-planar; see Section \ref{sec:asymp} for an explicit construction.

  \subsection{Curves and arcs}
  By a {\em curve} on a surface $S$ we mean the isotopy class of a
  simple closed curve on $S$. All curves will be {\em essential}, meaning they do not
  bound a disk or once-punctured disk, nor do they cobound an annulus with a component of $\partial S$. An {\em arc} is the isotopy class (relative
  to $\partial S$) of an embedded path that connects two boundary components of
  $S$. We say that two curves/arcs are {\em disjoint} if they can be
  isotoped off each other. As usual, we will not distinguish between curves (resp. arcs) and their representatives.

  \subsection{Mapping class groups}
  The mapping class group $\Map(S)$ of $S$ is the group of isotopy
  classes of orientation-preserving homeomorphisms of $S$; here all
  homeomorphisms and isotopies are assumed to fix the boundary of $S$
  pointwise. The {\em pure} mapping class group $\PMap(S)$ is the subgroup of
  $\Map(S)$ whose elements fix every end of $S$.
  When $S$ has finite type, then $\Map(S)$ is well-known to be of type
  $F_\infty$ \cite{Harv}.
  On the other hand, when $S$ has infinite type, then $\Map(S)$ is
  uncountable.

  When $S$ has infinite type, an important subgroup of $\Map(S)$ (and
  of $\PMap(S)$, in fact) is the {\em compactly-supported} mapping class
  group $\Map_c(S)$, which consists of those mapping classes that are
  the identity outside some compact subset of $S$. By a result of
  Patel-Vlamis \cite{PV}, $\Map_c(S)$ is dense (in the compact-open topology) inside $\PMap(S)$ if and
  only if $S$ has at most one non-planar end; otherwise, $\PMap(S)$ is
  topologically generated by $\Map_c(S)$ plus the set of {\em handle-shifts}.

 \subsection{Rigid structures and surface Houghton groups} 
  \label{sec:asymp}
  The goal of this subsection is to describe the  surfaces $\Sigma_n$ and the
  additional structure eluded to in the introduction necessary to define
  their asymptotic mapping class groups.  These definitions, along with
  the toolkit needed to present it, closely follows \cite[Section
  3]{ABF+21}.

  Fix an integer $n\ge 2$. Let $O_n$ be a sphere with $n$ boundary
  components, labelled $b_1, \ldots, b_n$, and $T$ a torus with two
  boundary components, denoted $\partial^- T$ and $\partial^+ T$; we
  refer to $\partial^+ T$ as the {\em top} boundary component of $T$. We
  fix, once and for all, an orientation-reversing homeomorphism
  $\lambda: \partial^-T \to \partial^+T$, and orientation-reversing
  homeomorphisms $\mu_i: \partial^-T \to b_i$, for $i=1, \ldots, n$.  We
  construct a sequence of compact, connected, orientable surfaces
  $(M^i)_i$ as follows:

  \begin{itemize}
    \item
      $M^1= O_n$;
    \item
      $M^2$ is the result of gluing a copy of $T$ to each boundary
      component of $O_n$, using the homeomorphisms $\mu_i$;
    \item
      For each $i\ge 3$, $M^i$ is the result of gluing a copy of $T$
      along each of the boundary components of $M^{i-1}$, using the
      homeomorphism $\lambda$.
  \end{itemize}

  The surface $\Sigma_n$ is the union of the
  surfaces $M^i$ above. The closure of each of the connected components of $M^i \setminus
  M^{i-1}$, for $i\ge 2$, is called a {\em piece}. By construction, each
  piece $B \subset M^i\setminus M^{i-1}$ is one of the glued on copies
  of $T$, and so is equipped with a canonical homeomorphism $i_B \colon
  B \to T$. We call the unique boundary component of $B$ that belongs
  to $M^{i-1}$ the {\em top} boundary component of $B$ as it maps by
  $i_B$ to $\partial_+T$. The set
  \[ \{\iota_B: B \text{ is a piece}\}\]
  is called the {\em rigid structure} on $\Sigma_n$.
	
    A subsurface of $\Sigma_n$ is {\em suited} if it is connected and is the union of $O_n$ and finitely many pieces. A boundary component of a suited subsurface is called a {\em suited curve}.

    Let $f: \Sigma_n \to \Sigma_n$ be a homeomorphism. We say that $f$
    is {\em asymptotically rigid} if there exists a suited subsurface
    $Z\subset \Sigma_n$, called a {\em defining surface} for $f$, such
    that:
    \begin{itemize}
      \item
        $f(Z)$ is a suited subsurface, and
      \item
        $f$ is {\em rigid away from} $Z$, that is, for every piece $B
        \subset \overline{\Sigma_n \setminus Z}$, we have that $f(B)$ is a piece, and
        $f|_{B} \equiv \iota_{f(B)}^{-1} \circ \iota_B$.
    \end{itemize}

    \begin{definition}[Surface Houghton group]
The surface Houghton group $\mathcal B_n$ is the subgroup of
  the mapping class group $\Map(\Sigma_n)$ whose elements have an
  asymptotically rigid representative.
    \end{definition}

 We will denote by
  $P\B_n$ the intersection of $\B_n$ with the pure mapping class group
  $\PMap(\Sigma_n)$.

\begin{remark}\label{rmk:subgroup}
In \cite{Funar07}, Funar described the braided Houghton groups ${\rm br}H_n$ as asymptotically rigid mapping class groups of certain planar surfaces. Replacing punctures of this planar surface with boundary components, and then doubling  the resulting surface, determines a homomorphism ${\rm br} H_n \to \mathcal B_n$ analogous to the construction Ivanov-McCarthy \cite[Section~2]{IM}.
\end{remark}

  \section{Finiteness properties}
   \label{sec:brown}
  \subsection{Brown's criterion for finiteness properties} 
 
  In this section we recall a classical criterion, due to Brown
  \cite{Br87}, for a group to (not) have certain finiteness
  properties. We remark that our formulation of the criterion differs
  from the original \cite{Br87}, as we use the language of 
  {\em discrete Morse theory} as developed in~\cite{BB97}.
  We recall here the basic notions and
  definitions, referring the interested reader to \cite{BB97} or
  \cite[Appendix A]{ABF+21} for a thorough discussion.

  Our setting is a group $\TheGroup$ acting on a piecewise euclidean
  CW-complex $\TheGcx$ by cell-permuting homeomorphisms that restrict to
  isometries on cells. A \notion{Morse function} on $\TheGcx$ is a cell-wise
  affine map
  \(
    \TheMorseFct\mapcolon
    \TheGcx\to\RealNumbers
  \)
  satisfying the condition that on each
  closed cell it attains a unique maximum (at a vertex, 
  the \notion{top vertex} of the cell). The \notion{descending link} 
  $\DescLinkOf{\TheVertex}$ of a vertex $\TheVertex$ is
  the part of its link spanned by the links of cells that contain $\TheVertex$ as their
  top vertex. The value $\TheMorseFctOf{\TheVertex}$ 
  is often referred to as the \notion{height} of $\TheVertex$.

  We consider the vertices the \notion{critical points} in 
  $\TheGcx$ and the images of vertices under $\TheMorseFct$ are 
  the \notion{critical values}. A Morse function
  \(
    \TheMorseFct\mapcolon
    \TheGcx
    \to
    \RealNumbers
  \)
  \notion{discrete} if the set of its
  critical values is a discrete subset of 
  $\RealNumbers$. Now the first
  theorem can be stated as follows.

  \begin{theorem}[Brown] \label{thm:brown}
    Let $\TheGroup$ be a group acting by cell-wise isometries on a
    contractible piecewise euclidean CW-complex $\TheGcx$. Assume 
    $\TheGcx$ is
    equipped with a discrete $\TheGroup$-invariant Morse function
    \(
      \TheMorseFct\mapcolon
      \TheGcx\to\RealNumbers
    \),
    and let $\TheGcx[][\le \TheCrVal]$ denote the largest subcomplex of 
    $\TheGcx$ fully
    contained in the preimage $\TheMorseFct[][-1]{(-\infty,\TheCrVal]}$. 
    Suppose that
    \begin{itemize}
      \item
        The quotient of $\TheGcx[][\le \TheCrVal]$ by $G$ is finite for 
        all critical values $\TheCrVal$. 
      \item
        Every cell stabilizer is of type \FType[\infty].
      \item
        There exists $\TheDim\ge 1$ such that, for sufficiently large 
        critical values $\TheCrVal$
        and for every vertex of $\TheVertex\in\TheGcx$ with 
        $\TheMorseFctOf{\TheVertex} \ge \TheCrVal$, the descending link of
        $\TheVertex$ in $\TheGcx$ is $\TheDim$-\notion{spherical}, i.e., 
        $(\TheDim-1)$-connected and of dimension $\TheDim$.
      \item
        For each critical value $\TheCrVal$, there exists a vertex $\TheVertex$
        with non-contractible descending link at height
        $\TheMorseFctOf{\TheVertex}\geq\TheCrVal$.
    \end{itemize}
    Then $\TheGroup$ is of type \FType[\TheDim] but not of type 
    \FPType[\TheDim+1].
  \end{theorem} 

  \begin{proof}
    This proof mimics Brown's proof of~\cite[Corollary~3.3]{Br87}.
    
    As the set of critical values is discrete, we can index them in order. Now,
    we consider the filtration of $\TheGcx$ by sublevel complexes 
    $\TheGcx[i] := \TheGcx[][{\le \TheCrVal[i]}]$
    where $\TheCrVal[i]$ runs through the critical values in order. 
    The filtration step $\TheGcx[i+1]$ is obtained from 
    $\TheGcx[i]$ up to homotopy by coning
    off descending links. Thus, the hypothesis that descending links are
    $(\TheDim-1)$-connected implies that the inclusion of $\TheGcx[i]$ into
    $\TheGcx[i+1]$ induces isomophisms in homotopy (and homology) groups up
    to dimension $\TheDim-1$ and an epimorphism in homotopy (and homology)
    in dimension $\TheDim$. As we assume the complex $\TheGcx$ to be 
    contractible, the isomophisms in dimensions below $\TheDim$ must
    eventually all be trivial maps. It follows that for sufficiently
    large $i$, the sublevel complex $\TheGcx[i]$ is $(\TheDim-1)$-connected.
    By~\cite[Propositions~1.1 and~3.1]{Br87} implies that $\TheGroup$ is of 
    type \FType[\TheDim].

    For the negative direction, we focus on the system
    \(
      \{\HomologyOf[\TheDim]{\TheGcx[i]}\}_{i}
    \)
    of homology groups in dimension $\TheDim$. We already observed that for
    $i$ large enough, the morphisms in the system are onto. Hence, the system
    can only be essentially trivial if the homology groups
    \(
      \HomologyOf[\TheDim]{\TheGcx[i]}
    \)
    vanish for all large enough $i$. If, however, at the transition from $\TheGcx[i]$
    to $\TheGcx[i+1]$, we encounter a non-contractible $(\TheDim-1)$-connected descending
    link of dimension $\TheDim$, this descending link will have non-trivial homology
    in dimension $\TheDim$. Hence the descending link contains a $\TheDim$-cycle.
    If this cycle was a boundary in $\TheGcx[i]$, coning off the descending link
    would provide another way of bounding it in $\TheGcx[i+1]$, thus creating a
    non-trivial $(\TheDim+1)$-cycle, which 
    is impossible as that element of $\Homology[\TheDim+1]$ could never be killed
    in the future as we only ever cone off $\TheDim$-dimensional links. Hence, $\HomologyOf[\TheDim]{\TheGcx[i]}\neq0$ for infinitely many $i$.
  \end{proof}

  It is clear from the criterion that tools for the analysis of connectivity
  properties of spaces can be useful. We collect the tools that we will need
  in this case.

  \subsection{Propagating connectivity properties}
  We use the convention that every space is $(-2)$-connected and that
  any non-empty space is $(-1)$-connected if it is non-empty.
  
  \subsubsection{Complete joins}
  We will deduce connectivity properties of complexes from those of
  other complexes, and maps between them.  One way to do this is
  formalized through the notion of a \notion{complete join complex},
  introduced by Hatcher and Wahl in \cite{HW10}.

  \begin{definition}[Complete join complex]\label{defn-join}
    Let $\TheJoinCx$ and $\TheBaseCx$ be simplicial complexes, and
    \(
      \TheJoinPr\mapcolon
      \TheJoinCx
      \to
      \TheBaseCx
    \)
    a
    simplicial map. 
    We say that $\TheJoinCx$ is a \notion{complete join
    complex} over $\TheBaseCx$ (with respect to the map $\TheJoinPr$)
    if the following properties are satisfied
    \begin{itemize}
      \item
        $\TheJoinPr$ is surjective, and is injective on individual simplices.
      \item
        For each simplex 
        \(
          \TheSimplex 
          = 
          \langle \TheVertex[0] ,\ldots, \TheVertex[\TheDim] \rangle
        \),
        its preimage can be written as a join of fibers over vertices:
        \[
          \TheJoinPrOf[][-1]{\TheSimplex}
          = 
          \TheJoinPrOf[][-1]{\TheVertex[0]} 
          \join \cdots \join
          \TheJoinPrOf[][-1]{\TheVertex[\TheDim]}
          .
        \] 
    \end{itemize} 
  \end{definition} 
  Since $\TheJoinPr$ is injective on simplices, the $\TheJoinPr$-preimages
  of vertices are discrete sets of vertices in $\TheJoinCx$. They are non-empty
  since $\TheJoinPr$ is surjective. It follows that
  for each $\TheDim$-simplex $\TheSimplex$ in $\TheBaseCx$, the $\TheJoinPr$-preimage
  $\TheJoinPrOf[][-1]{\TheSimplex}$ is a join of $\TheDim+1$ non-empty
  discrete sets.

  We make use of a complete join in two places, once to transfer
  known connectivity properties from $\TheBaseCx$ to $\TheJoinCx$,
  and once to go the
  other way. The first direction is the difficult one, and has been
  established by Hatcher--Wahl \cite{HW10}.
  
  Before stating the result, recall that a simplicial complex is \notion{weakly Cohen--Macaulay} 
  of dimension $\TargetConn+1$ if it is $\TargetConn$--connected and if the
  link of every
  simplex $\TheSimplex$ is $(\TargetConn-\DimOf{\TheSimplex}-1)$--connected.
  \begin{proposition}\cite[Proposition 3.5]{HW10} \label{prop-cjoin-conn}
    Suppose
    \(
      \TheJoinPr\mapcolon
      \TheJoinCx
      \to
      \TheBaseCx
    \)
    is a complete join. If $\TheBaseCx$ is weakly Cohen--Macaulay 
    of dimension $\TargetConn+1$, then so is $\TheJoinCx$.
  \end{proposition}

  Going forward is the easy direction (and is implicit in the argument given
  by Hatcher and Wahl \cite{HW10}). See \cite[Remark A.15]{ABF+21} for an explicit
  account.
  \begin{proposition}\label{prop:join-push-forward}
    Suppose
    \(      
      \TheJoinPr\mapcolon
      \TheJoinCx
      \to
      \TheBaseCx
    \)
    is a complete join. Then $\TheBaseCx$ is a retract of $\TheJoinCx$ and inherits
    all properties that can be expressed by the vanishing of group-valued functors
    or cofunctors. In particular, if $\TheJoinCx$ is $\TargetConn$-connected, then so
    is $\TheBaseCx$.
  \end{proposition}

  Similarly, complete joins are easily prevented from being contractible.
  \begin{observation}\label{obs:join-not-contractible}
    Let
     \(      
      \TheJoinPr\mapcolon
      \TheJoinCx
      \to
      \TheBaseCx
    \)
    be a complete join, and suppose that there is a top-dimensional simplex
    \(
      \TheSimplex
      =
      \langle \TheVertex[0] ,\ldots, \TheVertex[\TheDim] \rangle
    \)
    in $\TheBaseCx$ such that each vertex fiber 
    $\TheJoinPrOf[][-1]{\TheVertex[i]}$ contains at least two points.
    Then, the fiber over $\TheSimplex$ contains
    a $\TheDim$-sphere which defines a cycle in the homology of $\TheJoinCx$
    that cannot be a boundary as $\TheJoinCx$ does not contain simplices of 
    dimension $\TheDim+1$. In particular, $\TheJoinCx$ is not contractible.
  \end{observation}

  
  \subsubsection{The bad simplex argument}
  A map from a complete join is a particularly nice projection. The
  \emph{bad simplex argument}, introduced by Hatcher and Vogtmann \cite{HV17}, uses
  the inclusion map of a subcomplex together with additional local information
  to transfer connectivity properties from the ambient complex to the
  subcomplex. We do not follow the original exposition of the argument, since we find the language introduced in \cite{ABF+21} more convenient.
  
  Let $\TheSimplicialCx$ be a simplicial complex and assume that we are
  given a map $\TheSimplex\mapsto\RottenCoreOf{\TheSimplex}$ that assigns
  to each simplex $\TheSimplex$ in $\TheSimplicialCx$ a (possibly empty)
  face, $\RottenCoreOf{\TheSimplex}$, of $\TheSimplex$. We assume that
  the following two conditions are satisfied:
  \begin{itemize}
    \item
      Monotonicity:
      \(
      \TheSimplex\faceof\AltSimplex
      \,\,\Longrightarrow\,\,
      \RottenCoreOf{\TheSimplex} \faceof \RottenCoreOf{\AltSimplex}
      .
      \)
    \item
      Idempotence:
      \(
      \RottenCoreOf{\RottenCoreOf{\TheSimplex}}
      =
      \RottenCoreOf{\TheSimplex}
      .
      \)
  \end{itemize}
  We call the simplex $\TheSimplex$ \notion{good} if 
  $\RottenCoreOf{\TheSimplex}$ is empty and \notion{bad}
  if $\RottenCoreOf{\TheSimplex}=\TheSimplex$. Note that by monotonicity,
  the good simplices in $\TheSimplicialCx$ form a subcomplex
  \(
    \GoodScxOf{\TheSimplicialCx}
  \), 
  which we call the \notion{good subcomplex}.

  The \notion{good link} of a bad simplex $\BadSimplex$ is the geometric
  realization of the poset of those proper cofaces $\AltSimplex>\BadSimplex$
  for which $\RottenCoreOf{\AltSimplex}=\RottenCoreOf{\BadSimplex}=\BadSimplex$
  holds, i.e.,
  \[
    \GoodLinkOf{\BadSimplex}
    :=
    \{\,
    \AltSimplex>\BadSimplex
    \mid
    \RottenCoreOf{\AltSimplex} = \RottenCoreOf{\BadSimplex}
    \,\}
  \]  

  The following proposition is due to Hatcher--Vogtmann~\cite[Section~2.1]{HV17},
  but this wording is taken from~\cite[Proposition~A.7]{ABF+21}.
  \begin{proposition}\label{prop :connectivity of the goodlink}
    Assume that for some number $\TargetConn\geq -1$ and every bad 
    simplex $\TheSimplex$, the \notion{good link} $\GoodLinkOf{\TheSimplex}$
    is $(\TargetConn-\DimOf{\BadSimplex})$-connected. Then the inclusion
    $\GoodScxOf{\TheSimplicialCx}\monorightarrow\TheSimplicialCx$
    induces isomorphisms in $\pi_{\TheDim}$ for all $\TheDim\leq\TargetConn$
    and an epimorphism in $\pi_{\TargetConn+1}$.
  \end{proposition}

  
  \section{A contractible cube complex} 
  \label{sec:cubecx}
  The strategy for proving Theorem \ref{thm:finprop} is well-known,
  and is similar to that used in \cite{BFM+16,GLU20,ABF+21}, for
  instance. Namely, we will construct a 
  contractible complex $\mathfrak X$ on which $\mathcal B_n$ acts in such way that we can apply Brown's \cite{Br87} criterion described in Section
  \ref{sec:brown} above.

  We now proceed to do this. The definitions of the main objects are the word-for-word adaptations of the ones in \cite{ABF+21}, and the only differences with our situation will occur when we analyze the connectivity properties of
  descending links. For this reason, we will briefly recall the
  constructions from \cite[Section 5]{ABF+21}, referring to that article
  for a more thorough presentation.

  Consider ordered pairs $(Z,f)$, where $Z \subset \Sigma_n$ is a
  suited subsurface and $f \in \mathcal B_n$. Two such pairs
  $(Z_1,f_1)$ and $(Z_2,f_2)$ are said to be equivalent if $f_2^{-1}
  \circ f_1(Z_1) = Z_2$ and $f_2^{-1} \circ f_1$ is rigid in the
  complement of $Z_1$. Intuitively, the equivalence class of $(Z,f)$
  records the ``non-rigid" behavior of $f$ outside $Z$. For example, if
  $f \in \B_n$ is the identity outside a suited subsurface $Z$, then
  $(Z,f)$ is equivalent to $(Z,\id)$. As another useful example to keep
  in mind, observe that $(Z,f_1)$ and $(Z,f_2)$ are equivalent if
  $f_2^{-1} \circ f_1$ leaves $Z$ invariant and is rigid outside $Z$.

  We denote the equivalence class of $(Z,f)$ by $[Z,f]$, and the set
  of all equivalence classes by $\mathcal S$. The group $\B_n$ acts on
  $\mathcal S$, by setting
  \[
    g \cdot [Z,f] = [Z,g \circ f].
  \]

  We define the {\em complexity} of a pair $(Z,f)$ as above to
  be the genus of $Z$. Alternatively, observe that since $Z$ is a
  suited subsurface, it is the union of $O_n$ and some number of pieces.
  As each piece contributes $1$ to the genus, the complexity is
  simply the number of pieces in $Z$. Clearly, equivalent pairs have
  equal complexity, and the action preserves complexity, so we have a $\mathcal B_n$--invariant {\em complexity function}
  \[
    h \colon \mathcal S \to \mathbb Z \subset \mathbb R.
  \]

 Given vertices $x_1, x_2 \in \mathcal S$, we say that $x_1 \prec x_2$ if there are representatives $(Z_i,f_i)$ of $x_i$, for $i=1,2$, so that $f_1= f_2$, $Z_1 \subset Z_2$, and $\overline{Z_2 \setminus Z_1}$ is a disjoint union of pieces. We stress that $\prec$ is not a partial order.

 The relation $\prec$ can be used to construct a cube complex $\mathfrak X$ for which $\mathcal S$ is the $0$--skeleton. Given $x_1 \prec x_2$, the set $\{x \mid x_1 \preceq x \preceq x_2\}$ are the vertices of a $d$--cube,  with $d=h(x_2) - h(x_1)$.  We call $x_1$ the {\em bottom} of the cube, as it uniquely minimizes complexity over all its vertices. Since $\Sigma_n$ has exactly $n$ ends, the complex $\mathfrak X$ is $n$-dimensional. 
 
Observe that the action $\mathcal B_n$ on $\mathcal S$ preserves the cubical structure, and that the complexity function $h$ extends linearly over cubes to a $\B_n$-invariant complexity function (of the same name)
\[
h:\mathfrak X \to \mathbb R_+.
\]


For $k\ge 1$, write $\mathfrak X^{\le k}$ for the subcomplex of $\mathfrak X$ spanned by those vertices with complexity $\le k$. A direct translation of the arguments of Proposition 5.7 and Lemmas 6.2 and 6.3 of \cite{ABF+21} yields the following: 

\begin{theorem}\label{thm:brown-preparation}
The cube complex $\mathfrak X$ is contractible, and the action of $\mathcal B_n$ on $\mathfrak X$ satisfies: 
\begin{itemize}
    \item Let $C$ be a cube with bottom vertex $x=[Z,f]$. Then the $\mathcal B_n$-stabilizer of $C$ is isomorphic to a finite extension of $\Map(Z)$. In particular, every cube stabilizer is of type $F_\infty$.  
    \item For every $k\ge 1$, the quotient of $\mathfrak X^{\le k}$ by $\mathcal B_n$ is compact. 
\end{itemize}
\end{theorem}

In light of the theorem above, in order to apply Brown's Theorem \ref{thm:brown}, we need to prove that descending links have the correct connectivity properties. As was the case in \cite{ABF+21}, the connectivity properties of descending links are determined by those of {\em piece complexes}, whose definition we now recall:

 \begin{definition}[Piece complex]
    Let $\CmpSurface$ be a compact surface with boundary, and let 
    $\BdCmps$ be a collection of boundary circles. We define the
    \notion{piece complex} $\PieceCxOf{\CmpSurface,\BdCmps}$ to be 
    the simplicial subcomplex of the curve complex of $\CmpSurface$ 
    whose vertices are separating curves which, together with a 
    boundary circle from $\BdCmps$, bound a genus
    $1$ subsurface. If $Q = \partial Z$, we will write $\PieceCxOf{\CmpSurface,\BdCmps}=\PieceCxOf{\CmpSurface}$.
  \end{definition}

The relation between the two complexes is encapsulated by the following result, whose proof is exactly the same as that of \cite[Proposition 6.6
]{ABF+21}.

\begin{proposition}\label{prop:descending link
  to piece complex}
Let $x=[Z,\id]$ be a vertex of $\mathfrak X$. Then, the descending link $\lk^\downarrow(x)$ is a complete join over the piece complex $\mathcal P(Z)$.
\end{proposition}

  An immediate consequence of Propositions~\ref{prop:descending link
  to piece complex} and \ref{prop-cjoin-conn} is the following.
  \begin{corollary}\label{cor:piece-to-desc-link}
    If $\PieceCxOf{\CmpSurface}$ is weakly Cohen--Macaulay of dimension
    $\TargetConn$, then so is the descending link $\DescLinkOf{\CmpSurface}$.
  \end{corollary}

Before we end this section, we will need a little more information about this complete join for the proof of the negative part of Theorem~\ref{thm:finprop}.  To explain, we first recall that for $x = [Z,\id]$, the complete join map $\eta \colon \lk^\downarrow(x) \to \mathcal P(Z)$ is defined as follows.  Given $[W,g] \in \lk^\downarrow(x)$, there is a piece $Y \subset \Sigma_n$ so that $W \cup Y$ is a suited subsurface and $[W \cup Y,g] = x = [Z,\id]$.  It follows that $g(W \cup Y) = Z$ and $g(Y)$ is thus a vertex of $\mathcal P(Z)$.  We then define $\eta([W,g]) = g(Y)$.  With this, we now state the lemma we will need.

\begin{lemma} \label{lem:fibers infinite}
    For every vertex $X \in \mathcal P(Z)$, the fiber in $\lk^\downarrow(x)$ is infinite.
\end{lemma}
\begin{proof}  Given $X = \eta([W,g]) = g(Y)$ as above, we need to show that $\eta^{-1}(X)$ is infinite.  For this, we let $h \colon Y \to Y$ be any homeomorphism representing an element of $\Map(Y)$.  We extend $h$ by the identity outside $Y$ to a homeomorphism of the same name, which thus represents an element of $\mathcal B_n$.  Observe that $g \circ h (W \cup Y) = g(W \cup Y) = Z$ and $g \circ h(Y) = g(Y) \subset Z$, while $g \circ h$ is rigid outside $Z = W \cup Y$, so $[W,g \circ h] = [Z,\id]$, and thus $[W,g\circ h] \in \lk^\downarrow(x)$ with $\eta([W,g \circ h]) = g \circ h(Y) = g(Y) = X$.  That is, $[W,g \circ h] \in \eta^{-1}(X)$ is another vertex.  Moreover, $[W,g \circ h] = [W,g]$ if and only if $g^{-1} \circ g \circ h = h$ is rigid outside $W$ (up to isotopy).  Since $Y$ is a piece outside $W$, this can only happen if $h$ is isotopic (in $\Sigma_n$) to a homeomorphism which restricts to the identity in $Y$.  This is only possibly if the original homeomorphism $h$ of $Y$ represents the identity in $\Map(Y)$, modulo Dehn twisting in the essential component of $\partial Y$ in $Z$ (which can be ``absorbed" into $W$).  Since $\Map(Y)$ modulo the (central) subgroup generated by Dehn twisting in this component of $\partial Y$ is infinite, it follows that $\eta^{-1}(X)$ is infinite, as required.
\end{proof}

  \section{Connectivity properties of piece complexes}
  \label{sec:connectivity-of-the-piece-cx}
  \label{sec:firstpart}
  In this section, we shall establish connectivity properties
  of piece complexes and finally deduce the finiteness properties
  of $\OurGroup[\NumStrands]$.
  
  \begin{theorem}\label{thm:conn-of-piece-cx}
    The piece complex $\PieceCxOf{\CmpSurface,\BdCmps}$ of a compact
    surface is $\TargetConn$-connected, provided that
    $\GenusOf{\CmpSurface}\geq2\TargetConn+3$ and
    $\CardOf{\BdCmps}\geq\TargetConn+2$.
  \end{theorem}
  Let us first observe that this implies that piece complexes are Cohen--Macaulay.
  \begin{corollary}\label{cor:piece-cx-is-wcm}
    The piece complex $\PieceCxOf{\CmpSurface,\BdCmps}$ is weakly
    Cohen--Macaulay of dimension $\TargetConn+1$, provided that
    $\GenusOf{\CmpSurface}\geq2\TargetConn+3$ and
    $\CardOf{\BdCmps}\geq\TargetConn+2$.
  \end{corollary}
  \begin{proof}
    Observe that the link of a $\TheDim$-simplex
    $\TheSimplex$ in $\PieceCxOf{\CmpSurface,\BdCmps}$ is isomorphic
    to the piece complex $\PieceCxOf{\CutSurface,\StillAdmissible}$
    where $\CutSurface$ is obtained from $\CmpSurface$ by cutting off
    the pieces in $\TheSimplex$ and $\StillAdmissible$ is the set of
    those boundary circles in $\BdCmps$ that still exist 
    in $\CutSurface$.
    Then,
    \[
      \GenusOf{\CutSurface}
      =
      \GenusOf{\CmpSurface}-(\DimOf{\TheSimplex}+1) 
      \geq 2\TargetConn+3-\DimOf{\TheSimplex}-1 
      \geq 2(\TargetConn-\DimOf{\TheSimplex}-1)+3
    \]
    and
    \(
      \CardOf{\StillAdmissible}
      =
      \CardOf{\BdCmps}-(\DimOf{\TheSimplex}+1)
      \geq (\TargetConn-\DimOf{\TheSimplex}-1) +2
    \).
    This implies the link of $\sigma$ is $(\TargetConn-\DimOf{\TheSimplex}-1)$--connected, in view
    of Theorem~\ref{thm:conn-of-piece-cx}, as required.
  \end{proof}
  
  To analyze the connectivity properties of piece complexes, we shall
  introduce two more complexes: the \notion{injective tethered handle
  complex} $\THinjOf{\CmpSurface,\BdCmps}$, which we will see is a complete
  join over $\PieceCxOf{\CmpSurface, \BdCmps}$; and the \notion{tethered
  handle complex} $\THcxOf{\CmpSurface,\BdCmps}$, which contains
  the injective tethered handle complex as a subcomplex. As before, if
  $\BdCmps=\BoundaryOf{\CmpSurface}$ we will simply write 
  $\THinjOf{\CmpSurface,\BdCmps}= \THinjOf{\CmpSurface}$ and 
  $\THcxOf{\CmpSurface,\BdCmps}= \THcxOf{\CmpSurface}$. 
  A diagram of the maps we use reads as follows:
  \[
    \begin{tikzcd}[row sep=15]
      \DescLinkOf{[Z,f]}
      \ar[dr]
      &
      &
      \THinjOf{\CmpSurface}
      \ar[dl]
      \ar[r]
      &
      \THcxOf{\CmpSurface}
      \\
      &
      \PieceCxOf{\CmpSurface}
      &
      &
    \end{tikzcd}
  \]
  We have used the left arrow to pull back connectivity
  from the piece complex to the descending link. We shall use
  the middle arrow to push forward connectivity from
  the injective handle complex to the piece complex, and we will
  use the inclusion of $\THinjOf{\CmpSurface}$ into $\THcxOf{\CmpSurface}$ to apply a \emph{bad simplex
  argument}, pulling back connectivity from the tethered handle complex
  to the injective handle complex. The connectivity of the tethered handle
  complex itself has been analyzed in~\cite{ABF+21}.

  \subsection{Tethered handle complexes}
  Let $\TheSurface$ denote a compact connected orientable surface with
  boundary. By a
  \notion{handle} in $\TheSurface$ we mean a subsurface of $\TheSurface$
  that avoids the boundary $\BoundaryOf{\TheSurface}$ and is
  homeomorphic to a one-holed torus. 

  Given a handle  $\TheHandle$, consider (the isotopy class of) a simple arc
  $\TheTether$ that joins $\BoundaryOf{\TheHandle}$ to a
  component $\TheBdCmp \subset \TheSurface$. We remark that the isotopy class of $\TheTether$
  is not taken relative to its endpoints, as is sometimes the case. Observe
  that the regular neighborhood of $\TheHandle\cup\TheTether\cup\TheBdCmp$
  is a piece. We will refer to the union of $\TheHandle$ and $\TheTether$
  as a \notion{tethered handle} tethered to $\TheBdCmp$ with handle
  $\TheHandle$ and \notion{tether} $\TheTether$.

  \begin{definition}[Tethered handle complex]
    Let $\TheSurface$ be a compact orientable connected surface, and let 
    $\BdCmps$ be a collection of boundary circles of $\TheSurface$.
    The \notion{tethered handle complex} $\THcxOf{\TheSurface,\BdCmps}$ 
    is the simplicial complex whose $\TheDim$-simplices are sets of
    $\TheDim+1$ pairwise disjoint tethered handles, each tethered to an
    element of $\BdCmps$.
  \end{definition}

  The following is Lemma~10.12 of \cite{ABF+21}, which builds upon work of Hatcher-Vogtmann \cite{HV17}: 
  \begin{lemma} \label{lem:TH conn}
    The tethered handle complex $\THcxOf{\TheSurface,\BdCmps}$ is
    $\TargetConn$-connected, provided that $\BdCmps$ is not empty and $\GenusOf{\TheSurface}\ge 2\TargetConn+3$.   
  \end{lemma}

  We consider the following subcomplex of the tethered handle complex.

  \begin{definition}[Injective tethered handle complex] 
    The \notion{injective tethered handle complex} $\THinjOf{\TheSurface,\BdCmps}$
    is the subcomplex of $\THcxOf{\TheSurface,\BdCmps}$ consisting of those
    simplices where the involved handles are tethered to pairwise distinct
    boundary components in $\BdCmps$.
  \end{definition}

  The reason why we are interested in this subcomplex is that it is a complete
  join over the piece complex, which allows us to push forward its connectivity.
  Indeed, we already oberved that a small regular neighborhood of a tethered handle
  together with its boundary component is
  homeomorphic to a piece. In this way we obtain a simplicial map 
  \[
    \pi:
    \THinjOf{\TheSurface,\BdCmps} 
    \to
    \PieceCxOf{\TheSurface,\BdCmps}.
  \] 
  
  The following is
  \cite[Lemma 8.8]{ABF+21}, and the proof applies verbatim to this
  setting:
  \begin{proposition}\label{prop:inj-handle-cx-is-complete-join}
    The map
    \(
      \pi: \THinjOf{\TheSurface,\BdCmps}
      \to
      \PieceCxOf{\TheSurface,\BdCmps}
    \)
    is a
    complete join.
  \end{proposition}
  In particular, if $\THinjOf{\TheSurface,\BdCmps}$ is $\TargetConn$-connected,
  then so is $\PieceCxOf{\TheSurface,\BdCmps}$ by
  Proposition~\ref{prop:join-push-forward}.
  Thus, we have now reduced the problem to analyzing the connectivity
  properties of the injective tethered handle complex 
  $\THinjOf{\TheSurface,\BdCmps}$. 

  \begin{theorem}\label{thm:inj_tether_cx_conn}
    The injective tethered handle complex $\THinjOf{\TheSurface,\BdCmps}$ is
    $\TargetConn$-connected, provided that
    $\GenusOf{\TheSurface}\geq2\TargetConn+3$ and
    $\CardOf{\BdCmps}\geq\TargetConn+2$.
  \end{theorem}
  \begin{proof}
    We induct on $\TargetConn$ starting at $\TargetConn=-1$, which is
    trivial. For $\TargetConn\geq 0$, we
    use the bad simplex argument for the
    inclusion of $\THinjOf{\TheSurface,\BdCmps}$ into 
    $\THcxOf{\TheSurface,\BdCmps}$.
    
    Consider a simplex $\TheSimplex$ in $\THcxOf{\TheSurface,\BdCmps}$ and
    define $\RottenCoreOf{\TheSimplex}$ to consist of those tethered handles
    in $\TheSimplex$ that are tethered to the same boundary component as
    another handle in $\TheSimplex$. Note that $\THinjOf{\TheSurface,\BdCmps}$
    consists precisely of those simplices $\TheSimplex$ for which 
    $\RottenCoreOf{\TheSimplex}$ is empty, i.e., the good simplices of $\THcxOf{\TheSurface,\BdCmps}$.  See Figure~\ref{F:an example}.

        \begin{center}
    \begin{figure}[htb]
      \includegraphics[scale=0.7]{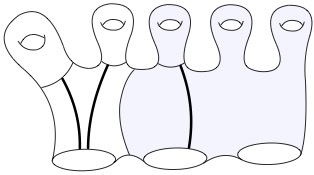}
      \caption{The figure illustrates an example of $Z$ where $g = 5$ and $Q = \partial Z$ has three components. There are three tethered handles, which define a $2$--simplex $\sigma$, while $\bar \sigma$ is a $1$--simplex spanned by the two tethered handles on the left.  The surface $Z'$ is the shaded subsurface, which has genus $3$ and three boundary components, but $Q'$ consists of just the two components on the right.}
      \end{figure} \label{F:an example}
    \end{center}
    
    Now consider a bad simplex $\BadSimplex=\RottenCoreOf{\BadSimplex}$.
    It takes at least two handles tethered to the same boundary component
    to make a bad simplex. Hence, $\DimOf{\BadSimplex}\geq 1$. Consequently,
    \(
      2 \DimOf{\BadSimplex}
      \geq 
      1+\DimOf{\BadSimplex}
    \)
    or, equivalently,
    \(
      \DimOf{\BadSimplex}
      \geq 
      \frac{1+\DimOf{\BadSimplex}}{2}.
    \)

    Note that the good link $\GoodLinkOf{\BadSimplex}$ is isomorphic to $\THinjOf{\CutSurface,\StillAdmissible}$, where $\CutSurface$ is 
    obtained from $\CmpSurface$ by cutting off, for each boundary circle
    used in $\BadSimplex$, a neighborhood of the boundary circle and all
    tethers and handles of $\BadSimplex$ that it meets. 
    The set $\StillAdmissible$ consists of
    those boundary circles in $\BdCmps$ that still exist 
    in $\CutSurface$. See Figure~\ref{F:an example}.
    Then, we note the following inequalities which allow us
    to apply the induction hypothesis.
    \begin{align*}
      \GenusOf{\CutSurface}
      &=
      \GenusOf{\TheSurface}-(\DimOf{\BadSimplex}+1)\\
      &\geq 
      \GenusOf{\TheSurface}-2\DimOf{\BadSimplex}\\
      &\geq 
      2\TargetConn+3-2\DimOf{\BadSimplex}\\
      &=
      2(\TargetConn-\DimOf{\BadSimplex})+3
    \end{align*}
    In a bad simplex each used boundary component tethers to
    at least two handles, whence we also obtain the following
    estimate:
    \begin{align*}
      \CardOf{\StillAdmissible} 
      & \geq 
      \CardOf{\BdCmps}-\frac{\DimOf{\BadSimplex}+1}{2}\\
      &\geq
      \CardOf{\BdCmps}-\DimOf{\BadSimplex}\\
      &\geq 
      \TargetConn + 2 -\DimOf{\BadSimplex}\\
      &=
      (\TargetConn-\DimOf{\BadSimplex})+2
    \end{align*}
    By the induction hypothesis, $\GoodLinkOf{\BadSimplex}$ is therefore
    $(\TargetConn-\DimOf{\BadSimplex})$-connected. 
    By Proposition \ref{prop :connectivity of the goodlink}, the inclusion
    \(
      \THinjOf{\TheSurface,\BdCmps} 
      \monorightarrow
      \THcxOf{\TheSurface,\BdCmps}
    \)
    induces isomorphisms in $\pi_{\TheDim}$
    for all $\TheDim\leq\TargetConn$. As $\THcxOf{\TheSurface,\BdCmps}$
    is $\TargetConn$-connected by Lemma~\ref{lem:TH conn}, 
    $\THinjOf{\TheSurface,\BdCmps}$ is $\TargetConn$-connected.
  \end{proof}

  \begin{proof}[Proof of Theorem~\ref{thm:conn-of-piece-cx}]
The result follows from the combination of Proposition \ref{prop:inj-handle-cx-is-complete-join} and Theorem \ref{thm:inj_tether_cx_conn}, appealing to Proposition~\ref{prop:join-push-forward}.
  \end{proof}

  \subsection{Proof of Theorem~\ref{thm:finprop}}

    The group $\OurGroup[\NumStrands]$ acts on $\TheCubeCx$ leaving
    the discrete Morse function
    \(
      \TheMorseFct
      \mapcolon
      \TheCubeCx
      \to
      \RealNumbers
    \)
    invariant. Sublevel complexes are $\OurGroup[\NumStrands]$-cocompact,
    and cell stabilizers are of type \FType[\infty] by 
    Theorem~\ref{thm:brown-preparation}.

    By the above analysis,
    descending links of vertices are of dimension $\NumStrands-1$, and
    for vertices of height $\TheMorseFct \geq 2\NumStrands$,
    they are $(\NumStrands-2)$-connected.
    By Proposition~\ref{prop:descending link to piece complex} 
    and Lemma~\ref{lem:fibers infinite}, the descending 
    links are complete joins with infinite vertex fibers over a complex of
    finite dimension, hence they are not contractible by
    Observation~\ref{obs:join-not-contractible}.

    Now the hypotheses of Brown's criterion~\ref{thm:brown} have been verified
    and the group $\OurGroup[\NumStrands]$
    is of type \FType[\NumStrands-1] but not of type \FPType[\NumStrands].

  \section{Abelianization}
  \label{sec:abelianization} 
  In this section we prove Theorem \ref{abelianizations}. The
  arguments rely on construction of non-trivial, integer-valued
  homomorphisms from pure mapping class groups of \cite{APV20}, which we
  briefly recall here for the particular case of the $n$-pronged
  surfaces $\Sigma_n$.

  Let $\gamma$ be an oriented curve that separates an end $E$ of
  $\Sigma_n$ from the rest, where the orientation is chosen so that the
  component of $\Sigma_n-\gamma$ to the right of $\gamma$ is a
  neighborhood of $E$. Then, $\gamma$ defines a nonzero element
  $[\gamma] \in H_1^{sep}(S,\mathbb Z)$, the subgroup of $H_1(S,\mathbb
  Z)$ spanned by homology classes of separating curves.

  To every $f \in \PMap(S)$ and $\gamma \in H^{sep}_1 ( S, \mathbb
  Z)$, the authors of \cite{APV20} associate an integer
  $\varphi_{[\gamma]}(f)$, which can be considered as the ``signed
  genus" between $\gamma$ and $f(\gamma)$ (see Section 3 in \cite{APV20}
  for the details). They then show that the map \[\varphi_{[\gamma]}
    \colon \PMap(S) \to \mathbb Z\] so obtained is a well-defined,
  nontrivial homomorphism which depends only on the homology class
  $[\gamma]$, see \cite[Proposition
  3.3]{APV20}. Furthermore, by \cite[Proposition~4.4]{APV20}, this
  induces a surjective homomorphism
  \[
    \Phi \colon \PMap(S) \to H^1_{sep}(S,\mathbb Z),
  \]
  by the rule $\Phi(f)([\gamma]) = \varphi_{[\gamma]}(f)$, where
  $H^1_{sep}(S, \mathbb Z) $ has been identified with
  $\text{Hom}(H_1^{sep}(S, \mathbb Z), \mathbb Z)$ by the Universal
  Coefficient Theorem. Informally speaking, $\Phi(f)$ tells us ``how
  much genus has been shifted" on each end by $f$. It follows from (the
  proof of) \cite[Theorem 5]{APV20} that $\Phi$ is a surjective
  homomorphism whose kernel is precisely $\overline{\PMap_c(S)}$, namely
  the closure of the compactly-supported mapping class group. At this
  point, in our particular setting, one has:

  \begin{theorem}[Corollary 6 of \cite{APV20}]\label{structure result}
    \(
      \PMap(\Sigma_n)
      =
      \overline{\PMap_c(\Sigma_n)}\rtimes \mathbb Z^{n-1}
    \).
  \end{theorem}

  When we restrict $\Phi$ to the pure subgroup $P\B_n$, the image of $\Phi$ is generated by $n-1$ handle
  shifts $\rho_1, \dots, \rho_{n-1}$, which belong to $P\B_n$.  Even
  though we do not have the semi-direct product structure in Theorem
  \ref{structure result} when restricting to $P\B_n$, the projection onto $\mathbb Z^{n-1}$ from that theorem still defines a surjective homomorphism fitting into a short exact sequence:
  \[
    1
    \rightarrow
    {(\B_n)}_c
    \rightarrow
    P\B_n
    \rightarrow
    \mathbb Z^{n-1}
    \to 1,
  \]
  where ${(\B_n)}_c$ is the intersection of $\B_n$ with
  $\overline{\PMap_c(\Sigma_n)}$, which is precisely the compactly supported elements of $\B_n$. Since ${(\B_n)}_c$ is a direct limit of mapping class
  groups of compact surfaces, Powell's result \cite{Po} implies that
  $P{(\B_n)}_c^{ab}=\{0\}$, and therefore $P\B_n^{ab} \cong \mathbb
  Z^{n-1}$. At this point, the fact that $\B_n^{ab} = \{0\}$ follows
  from the above and the natural short exact sequence
  \[
    1 \to P\B_n \to \B_n \to \mathbb \Sym(n) \to 1,
  \]
  where $\Sym(n)$ is the symmetric group on $n$ elements, and $\B_n \to \Sym(n)$ comes from the action on the $n$ ends of $\Sigma_n$. This is because the action can be used to conjugate generators of $P\B_n^{ab}$ to their inverses.  This proves Theorem~\ref{abelianizations}.

\bigskip

We observe that $(\B_n)_c$ is a normal subgroup of $\B_n$, since the conjugate of a compactly supported homeomorphism is compactly supported.  From the homomorphisms above, the quotient $G = \B_n/(\B_n)_c$ admits a homomorphism to $\Sym(n)$ with kernel $\mathbb Z^{n-1}$.  It is not hard to explicitly construct a splitting of the associated short exact sequence proving that $G \cong \mathbb Z^{n-1} \rtimes \Sym(n)$.  We thus have a short exact sequence
  \[ 1 \to (\B_n)_c \to \B_n \to (\mathbb Z^{n-1} \rtimes \Sym(n)) \to 1.\]
  On the other hand, the proof of \cite[Theorem~4.6]{PV} of Patel and Vlamis (which itself relies on a result of Paris \cite{Paris}) can be applied verbatim to show that $(\B_n)_c$ has no nontrivial finite quotients, proving the following.
  \begin{proposition} \label{prop:quotients}
      Every finite quotient of $\B_n$ factors through the homomorphism to $\mathbb Z^{n-1} \rtimes \Sym(n)$. \qed
  \end{proposition}

An application of this and Theorem~\ref{thm:finprop} proves Corollary \ref{cor:braidedsurface}:

\begin{proof}[Proof of Corollary \ref{cor:braidedsurface}]
   First observe that if $n\neq m$, then Theorem~\ref{thm:finprop} and \cite[Theorem~5.24]{GLU20} imply that the finiteness properties of the groups are different, and so they cannot be commensurable.

   If $n=m$, and $\B_n$ and ${\rm br}H_n$  are commensurable, then after passing to the intersection with the finite index pure subgroups, we find finite index subgroups $K < P\B_n$ and $J < P{\rm br}H_n$ so that $K \cong J$.  We note here that $P{\rm br}H_n$ is the kernel of the corresponding action on the $n$ (non-isolated) ends of the underlying surface, and is not the subgroup consisting of pure braids. 
   
   Applying Proposition~\ref{prop:quotients}, we see that $K^{ab} \cong \mathbb Z^{n-1}$, and hence $J^{ab} \cong \mathbb Z^{n-1}$ as well.  That is, both abelianizations must simply be the restrictions of the abelianization of $P\B_n$, and thus also $P{\rm br}H_n$, respectively.  The kernels of the abelianizations must therefore be finite index subgroups of $(\B_n)_c$ and $({\rm br}H_n)_c$, respectively.  The former has no finite index subgroups, whereas $({\rm br}H_n)_c$ admits a homomorphism to $\mathbb Z$ (being the direct limit of compactly supported braid groups), which therefore has infinitley many nontrivial finite index subgroups.  This contradiction proves that $\B_n$ and ${\rm br}H_n$ are not commensurable. 
\end{proof}

%
%

\section{Marking graphs}\label{S:markings}

A {\em marking} $\mu$ on a surface is a pants decomposition called the {\em base} of the marking, $\base(\mu)=\bigcup_i \alpha_i$, together with a choice of transverse curves $\beta_i$ for each $\alpha_i$; that is, a curve $\beta_i$ so that $i(\alpha_i,\beta_j) = 0$ if $i \neq j$ and $i(\alpha_i,\beta_i) = 1$ or $2$, depending on whether $\alpha_i$ and $\beta_i$ fill a one-holed torus or four-holed sphere, respectively.  Masur and Minsky \cite{MM00} define a graph whose vertices are markings and so that two markings differ by {\em elementary moves}, which essentially swaps the roles of $\alpha_i$ and $\beta_i$ (together with a certain ``clean up" operation to ensure the result is again a marking).  Let $\mathcal M_n$ denote the marking graph of $\Sigma_n$.  The image of a marking under a mapping class is again a marking, and the definition of elementary move implies that the mapping class group acts on the marking graph.

For a surface $S$ of finite type, its marking graph $\mathcal M(S)$ is locally finite and the orbit map is a quasi-isometry, since the action is cocompact.  However, the action of $\mathcal B_n$ on an invariant component of $\mathcal M_n$ is not cocompact since one can find arbitrarily many distinct homeomorphism types of markings.  Moreover, the orbit map is not even a quasi-isometric embedding.  Indeed, if $\mu$ is a marking, and $t_i$ is the Dehn twist in $\alpha_i \in \base(\mu)$, then the distance from $\mu$ to $t_i(\mu)$ is uniformly bounded, while the distance from the identity to $t_i$ tends to infinity as $i \to \infty$ (since these are all distinct elements in a finitely generated group).

To prove Theorem~\ref{T:sub marking graph}, we need the following.

\begin{lemma}
For all $n \geq 2$ and any marking $\mu$, the stabilizer of $\mu$ in $\mathcal B_n$ is either finite, or contains a finite index subgroup that acts on $\Sigma_n$ by covering transformations.  In particular, the stabilizer is finite if $n \geq 3$. 
\end{lemma}
\begin{proof}
There is a hyperbolic metric on $\Sigma_n$ so that all $\alpha_i \in \base(\mu)$ have length $1$ and all $\beta_i$ meet $\alpha_i$ orthogonally.  Then any element of the stabilizer of $\mu$ acts on $\Sigma_n$ by isometries.  The lemma now follows since the isometry group of $\Sigma_n$ is necessarily discrete, and hence finite for $n \geq 3$, and either finite or virtually cyclic acting by covering transformations for $n = 2$.
\end{proof}
It is straightforward to construct $\mathcal B_n$--invariant components of $\mathcal M_n$, for all $n \geq 2$ to which the following theorem applies, and which immediately implies Theorem~\ref{T:sub marking graph}.
\begin{theorem}
For any $\mu$ in a $\mathcal B_n$--invariant component $\mathcal M_n^0 \subset \mathcal M_n$ with finite stabilizer, there is a locally finite subgraph $\mathcal X \subset \mathcal M_n^0$ containing $\mu$ so that $\mathcal B_n$ acts properly and cocompactly on $\mathcal G$.
\end{theorem}
\begin{proof}
We let $\mathcal G$ be a finite subgraph which is the union of paths connecting $\mu$ to its image under each generator from some fixed finite generating set for $\mathcal B_n$.  Further, we assume that each vertex in $\mathcal G$ has finite stabilizer as well.  This is automatic for $n \geq 3$, and is easy to arrange for $n=2$.  Now set $\mathcal X = \mathcal B_n \cdot \mathcal G$.

The fact that $\mathcal B_n$ acts cocompactly on $\mathcal X$ is immediate, since the $\mathcal G$--translates cover $\mathcal X$ by construction.  The only thing we must prove is that $\mathcal X$ is locally finite.   For this, it suffices to show that 
\[ K = \{g \in \mathcal B_n \mid g \cdot \mathcal G \cap \mathcal G  \neq \emptyset \}\]
is finite.  Suppose there exists an infinite sequence of distinct elements $\{g_n\}\subset K$.  Let $x_n \in \mathcal G$ be a vertex so that $y_n = g_n \cdot x_n \in \mathcal G$.  There are only finitely many vertices of $\mathcal G$, and so after passing to a subsequence (and reindexing), $x_n=x$ and $y_n=y$ for some $x,y \in \mathcal G$.  Thus, $g_1^{-1} g_n \cdot x = x$ for all $n$, and hence the stabilizer of $x$ is infinite, a contradiction.
\end{proof}

\begin{remark}
    The utility in proving that $\Map(S)$ is quasi-isometric to the marking graph for a finite type surface $S$ is that it allows one to provide a coarse estimate for distances in terms of local information via Masur and Minsky's subsurface projections and their distance formula \cite{MM00}.  The Dehn twisting examples above imply that one cannot expect a similar distance formula for $\B_n$.  However, one may wonder if there is some restricted set of subsurfaces for which one can prove a distance formula.  Or perhaps there is still a distance formula for all of $\mathcal M_n$ (which simply does not transfer to $\B_n$ because it is not quasi-isometric to $\mathcal M_n$).
\end{remark}
	
  \bibliographystyle{amsplain}

\end{document}